\documentclass[12pt]{amsart}
\usepackage{amsmath, amssymb, amsthm}

\title[On the~cyclic subgroup separability]{On the~cyclic subgroup separability of~the~free product of~two groups with~commuting subgroups}

\author{E.~V.~Sokolov}

\newtheoremstyle{theorem}{10pt}{10pt}{\it}{\parindent}{\bf}{. }{ }{}
\theoremstyle{theorem}
\newtheoremstyle{corollary}{10pt}{10pt}{\it}{\parindent}{\bf}{. }{ }{}
\theoremstyle{corollary}
\newtheoremstyle{proposition}{10pt}{10pt}{\it}{\parindent}{\bf}{. }{ }{}
\theoremstyle{proposition}
\newtheoremstyle{lemma}{10pt}{10pt}{\it}{\parindent}{\bf}{. }{ }{}
\theoremstyle{lemma}
\newtheorem{theorem}{Theorem}[section]
\newtheorem{corollary}[theorem]{Corollary}
\newtheorem{proposition}[theorem]{Proposition}
\newtheorem*{lemma}{Lemma}

\begin{document}

\begin{abstract}
Let $G$ be the~free product of~groups $A$ and~$B$ with~commuting subgroups $H\leqslant A$ and~$K\leqslant B$, and~let $\mathcal{C}$ be the~class of~all finite groups or~the~class of~all finite \hbox{$p$-}groups. We derive the~description of~all $\mathcal{C}$-separable cyclic subgroups of~$G$ provided~this group is residually a~$\mathcal{C}$-group. We prove, in~particular, that if~$A$, $B$ are finitely generated nilpotent groups and~$H$, $K$ are $p^\prime$-isolated in~the~free factors, then all $p^\prime$-isolated cyclic subgroups of~$G$ are separable in~the~class of~all finite \hbox{$p$-}groups. The same statement is true provided~$A$, $B$ are free and~$H$, $K$ are $p^\prime$-isolated and~cyclic.
\end{abstract}

\maketitle

\section*{Introduction}
\footnotetext{\textit{Key words and~phrases:} subgroup separability, residual \hbox{$p$-}finiteness, free product of~two groups with~commuting subgroups.}
\footnotetext{2010 \textit{Mathematics Subject Classification:} 20E26, 20E06.}

Let $\pi$ be a~set of~prime numbers, and~let $\mathcal{F}_{\pi}$ be the~class of~all finite \hbox{$\pi$-}groups (recall that a~finite group is said to~be a~\textit{\hbox{$\pi$-}group} if,~and~only if,~all prime divisors of~its order belong to~$\pi$). It is the~main aim of~this paper to~investigate the~\hbox{$\mathcal{F}_{\pi}$-}separability of~cyclic subgroups of~free products of~two groups with~commuting subgroups.

Recall (see~\cite[section~4.2]{li01}) that \textit{the~free product of~two groups~$A$ and~$B$ with~commuting subgroups $H \leqslant A$ and~$K \leqslant B$},
$$
G = \langle A * B;\ [H,K] = 1 \rangle,
$$
is the~quotient group of~the~ordinary free product~$A * B$ by~the~mutual commutant $[H,K]$ of~$H$ and~$K$.

Certain algorithmic problems are investigated in~respect to~this construction in~\cite{li02, li04, li03}. E.~D.~Loginova \cite{li06, li05} researches its residual properties and~obtains a~condition which is necessary and~sufficient for~$G$ to~be residually an~\hbox{$\mathcal{F}_{\pi}$-}group provided~$\pi$ either coincides with~the~set of~all prime numbers or~is an~one-element set. Also, she proves that if~all cyclic subgroups of~$A$ and~$B$ are separable in~the~class of~all finite groups, then $G$ possesses the~same property. At~last, D.~Tieudjo and~D.~I.~Moldavanskii \cite{li08, li07, li09, li10} consider in~detail the~special case of~the~given construction when~$A$ and~$B$ are cyclic.

This paper generalizes the~results of~E.~D.~Loginova and~D.~Tieudjo mentioned above. We derive the~description of~all \hbox{$\mathcal{F}_{\pi}$-}separ\-able cyclic subgroups of~$G$ provided~$\pi$ either coincides with~the~set of~all prime numbers or~is an~one-element set and~not necessarily all cyclic subgroups of~the~free factors are \hbox{$\mathcal{F}_{\pi}$-}separ\-able.

The~first section of~the~paper is devoted to~the~discussion of~the~relationship between the~notions of~subgroup separability, subgroup isolatedness and~quasi-regu\-lari\-ty. The~second section contains the~proof of~the~main theorem. Some corollaries and~applications of~this result are given in~the~third section of~the~paper.

\section{Separability, isolatedness and~quasi-regu\-lari\-ty}

Let $X$ be a~group, and~let $\Omega$ be a~family of~normal subgroups of~$X$. We shall say that a~subgroup~$Y$ of~$X$ \textit{is separable by}~$\Omega$ if
$$
\bigcap_{N \in \Omega} YN = Y.
$$

For each subgroup~$Y$, we call the~subgroup
$$
\bigcap_{N \in \Omega} YN
$$
the~\textit{$\Omega$-closure} of~$Y$ and~denote it by~$\Omega$\hbox{-\textit{Cl}}$(Y)$. It is easy to~see that the~$\Omega$-closure of~$Y$ is the~least subgroup which contains~$Y$ and~is separable by~$\Omega$.

Let further $\Omega_{\pi}(X)$ denote the~family of~all normal subgroups of~finite \hbox{$\pi$-}index of~$X$ (i.~e. the~family of~all such subgroups that the~quotient groups by~them belong to~$\mathcal{F}_{\pi}$). If a~subgroup~$Y$ is separable by~$\Omega_{\pi}(X)$, it is called \textit{\hbox{$\mathcal{F}_{\pi}$-}separ\-able} in~$X$. A~group~$X$ is said to~be \textit{residually an~\hbox{$\mathcal{F}_{\pi}$-}group} if~its trivial subgroup is \hbox{$\mathcal{F}_{\pi}$-}separ\-able. If~$\pi$ coincides with~the~set of~all prime numbers and~$\mathcal{F}_{\pi}$ is the~class of~all finite groups respectively, we get the~well known notions of~\textit{(finite) subgroup separability} and~\textit{residual finiteness}.

We shall say also that a~group~$X$ is \textit{\hbox{$\mathcal{F}_{\pi}$-}quasi-regu\-lar} by~its subgroup~$Y$ if, for~each subgroup $M \in \Omega_{\pi}(Y)$, there exists a~subgroup $N \in \Omega_{\pi}(X)$ such that $N \cap Y \leqslant M$. The~notion of~quasi-regu\-lari\-ty is closely tied to~separability as~the~next statement shows.

\begin{proposition}\label{prop11}
Let $X$ be a~group, and~let $Y$ be its \hbox{$\mathcal{F}_{\pi}$-}separ\-able subgroup. Then $X$ is \hbox{$\mathcal{F}_{\pi}$-}quasi-regu\-lar by~$Y$ if, and~only if, all subgroups of~$\Omega_{\pi}(Y)$ are \hbox{$\mathcal{F}_{\pi}$-}separ\-able in~$X$.
\end{proposition}

\begin{proof}
\textit{Necessity.} Let $M$ be a~subgroup of~$\Omega_{\pi}(Y)$, and~let $x \in X$ be an~arbitrary element which does not belong to~$M$. If~$x \notin Y$, we can use the~\hbox{$\mathcal{F}_{\pi}$-}separability of~$Y$ and~find a~subgroup $N \in \Omega_{\pi}(X)$ such that $x \notin YN$ and, hence, $x \notin MN$.

Let now $x \in Y$. Since $X$ is \hbox{$\mathcal{F}_{\pi}$-}quasi-regu\-lar by~$Y$, there exists a~subgroup $N \in \Omega_{\pi}(X)$ satisfying the~condition $N \cap Y \leqslant M$. It is easy to~see that in~this case $MN \cap Y=M$ and~so again $x \notin MN$.

Thus, $M$ is \hbox{$\mathcal{F}_{\pi}$-}separ\-able in~$X$.

\textit{Sufficiency.} Let again $M$ be an~arbitrary subgroup of~$\Omega_{\pi}(Y)$, and~let $1=y_{1}, \ldots, y_{n}$ be a~set of~representatives of~all cosets of~$M$ in~$Y$. For~$M$ is \hbox{$\mathcal{F}_{\pi}$-}separ\-able in~$X$, we can find a~subgroup $N \in \Omega_{\pi}(X)$ such that $y_{2}, \ldots, y_{n} \notin MN$.

If~we suppose now that the~intersection $N \cap Y$ is not contained in~$M$, choose an~element $g \in (N \cap Y)\backslash M$, and~write it in~the~form $g=xy_{i}$ for~suitable $i \in \{2, \ldots, n\}$ and~$x \in M$, then we get $xy_{i} \in N$. But~it follows from here that $y_{i} \in MN$ what contradicts the~choice of~$N$. Thus, $N \cap Y \leqslant M$ and~$X$ is \hbox{$\mathcal{F}_{\pi}$-}quasi-regu\-lar by~$Y$.
\end{proof}

We call a~subgroup~$Y$ of~a~group~$X$ \textit{subnormal} in~this group if~there exists a~sequence of~subgroups
$$
Y=Y_{0} \leqslant Y_{1} \leqslant \ldots \leqslant Y_{n}=X
$$
every term of~which is normal in~the~next.

\begin{proposition}\label{prop12}
If~$Y$ is a~subnormal subgroup of~finite \hbox{$\pi$-}index of~a~group~$X$, then the~\hbox{$\mathcal{F}_{\pi}$-}quasi-regu\-lari\-ty of~$X$ by~$Y$ holds.
\end{proposition}

\begin{proof}
Let $M$ be an~arbitrary subgroup of~$\Omega_{\pi}(Y)$. Since $Y$ is a~subnormal subgroup of~finite \hbox{$\pi$-}index of~$X$, then $M$ possesses the~same properties. Hence, there exists a~series
$$
M=M_{0} \leqslant M_{1} \leqslant \ldots \leqslant M_{n}=X
$$
such that $M_{i} \in \Omega_{\pi}(M_{i + 1})$, $i=0, 1, \ldots , n-1$. Let us show that there is a~subgroup $N \in \Omega_{\pi}(X)$ such that $N \leqslant M$ and, therefore, $N \cap Y \leqslant M$.

We shall use induction on~$n$.

With~$n=1$ we have $M \in \Omega_{\pi}(X)$ and~so can simply put $N=M$.

Let now $n>1$ and
$$
N_1 = \bigcap _{x \in M_2 } M^x.
$$

It is obvious that $N_{1}$ is normal in~$M_{2}$. Consider a~set of~representatives of~all right cosets of~$M$ in~$M_{2}$: $\{x_{1}, x_{2}, \ldots, x_{q}\}$.

Since any element $x \in M_{2}$ can be written in~the~form $x=m_{x}x_{i}$ for~suitable $i \in \{1, \ldots, q\}$ and~$m_{x} \in M$, then $M^{x}=M^{x_i}$ and
$$
N_1 = \bigcap_{i = 1}^q M^{x_i}.
$$

For~$M_{1}$ is normal in~$M_{2}$, so it is invariant under the~conjugation by~any element~$x_{i}$. From this it follows that all subgroups~$M^{x_i}$ belong to~$\Omega_{\pi}(M_{1})$. Hence, the~quotient group $M_{1}/N_{1}$ is isomorphic to~a~subgroup of~the~direct product of~the~finite \hbox{$\pi$-}groups $M_{1}/M^{x_i}$ and~so is itself a~finite \hbox{$\pi$-}group.

Thus, $N_{1}$ belongs to~$\Omega_{\pi}(M_{2})$, and~we can apply the~induction hypothesis to~it. In accordance with~this hypothesis there exists a~subgroup $N \in \Omega_{\pi}(X)$ such that $N \leqslant N_{1}$. Since $N_{1} \leqslant M$, $N$ is desired, and~the~proposition is proved.
\end{proof}

As~usual, we denote by~$\pi^\prime$ the~complement of~a~set~$\pi$ in~the~set of~all prime numbers. Recall that a~subgroup~$Y$ of~a~group~$X$ is said to~be \textit{\hbox{$\pi^\prime$-}iso\-lated} in~$X$ if, for~each element $x \in X$ and~for each number $q \in \pi^\prime$, the~inclusion $x^{q} \in Y$ implies $x \in Y$.

\begin{proposition}\label{prop13}
Every \hbox{$\mathcal{F}_{\pi}$-}separ\-able subgroup is \hbox{$\pi^\prime$-}iso\-lated.\linebreak In~particular, if~a~group is residually an~\hbox{$\mathcal{F}_{\pi}$-}group, then it is \hbox{$\pi^\prime$-}torsion-free.
\end{proposition}

\begin{proof}
Let us suppose that a~subgroup~$Y$ of~a~group~$X$ is not \hbox{$\pi^\prime$-}iso\-lated and~$x \in X$ is such an~element that $x \notin Y$ but $x^{q} \in Y$ for~some number $q \in \pi^\prime$. Let also $N$ be an~arbitrary subgroup of~$\Omega_{\pi}(X)$, and~let $n$ be the~order of~$x$ modulo this subgroup.

Since $n$ is a~\hbox{$\pi$-}num\-ber, there exists a~natural number~$m$ such that $qm \equiv 1 \pmod n$ and, hence, $x \equiv x^{qm} \pmod N$. Then $x \in YN$ and, because $N$ has selected arbitrarily, $Y$ is not \hbox{$\mathcal{F}_{\pi}$-}separ\-able in~$X$.
\end{proof}

Let us call the~least \hbox{$\pi^\prime$-}iso\-lated subgroup of~a~group $X$ containing a~subgroup $Y$ the~\textit{\hbox{$\pi^\prime$-}iso\-lator} of~$Y$ in~$X$ and~denote it by~$\mathcal{I}_{\pi^\prime}(X, Y)$.

\begin{proposition}\label{prop14}
If~$X$ is residually an~\hbox{$\mathcal{F}_{\pi}$-}group, then the~\hbox{$\pi^\prime$-}iso\-lator of~an~arbitrary abelian subgroup~$Y$ of~$X$ is an~abelian subgroup and~coincides with~the~set of~all \hbox{$\pi^\prime$-}roots being extracted from the~elements of~$Y$ in~$X$.
\end{proposition}

\begin{proof}
Since the~$\Omega_{\pi}(X)$-closure of~$Y$\,is\,\hbox{$\mathcal{F}_{\pi}$-}separ\-able,\,then\,it\,is\,a\,\hbox{$\pi^\prime$-}iso\-lated subgroup of~$X$ containing~$Y$. Therefore,
$$
\mathcal{I}_{\pi^\prime}(X, Y) \leqslant \Omega_{\pi}(X)\text{-\textit{Cl}}(Y).
$$

Suppose that the~commutator of~some elements $x_{1}, x_{2} \in \mathcal{I}_{\pi^\prime}(X, Y)$ is not equal to~$1$. Then, since $X$ is residually an~\hbox{$\mathcal{F}_{\pi}$-}group, there exists a~subgroup $N \in \Omega_{\pi}(X)$ such that $[x_{1}, x_{2}] \notin N$. From the~other hand, $x_{1}, x_{2} \in \Omega_{\pi}(X)$\hbox{-\textit{Cl}}$(Y)$ in~view of~the~inclusion mentioned above, and, in~particular, $x_{1}, x_{2} \in YN$. So the~elements $x_{1}N$, $x_{2}N$ belong to~the~abelian subgroup~$YN/N$ of~$X/N$ and~$[x_{1}N, x_{2}N]=1$. It follows from here that $[x_{1}, x_{2}] \in N$, and~we get a~contradiction proving that $\mathcal{I}_{\pi^\prime}(X, Y)$ is an~abelian group.

Let now $u, v \in X$ be arbitrary elements such that $u^{q}, v^{r} \in Y$ for~some \hbox{$\pi^\prime$-}num\-bers~$q$ and~$r$. Then $u, v \in \mathcal{I}_{\pi^\prime}(X, Y)$ and~$(uv)^{qr}=u^{qr}v^{qr} \in Y$. Therefore, the~set of~all \hbox{$\pi^\prime$-}roots being extracted from the~elements of~$Y$ in~$X$ is a~subgroup and, hence, coincides with~$\mathcal{I}_{\pi^\prime}(X, Y)$.
\end{proof}

\begin{proposition}\label{prop15}
If~$X$ is residually an~\hbox{$\mathcal{F}_{\pi}$-}group, then the~\hbox{$\pi^\prime$-}iso\-lator of~an~arbitrary locally cyclic subgroup~$Y$ of~$X$ is a~locally cyclic subgroup.
\end{proposition}

\begin{lemma}
If~$x, y \in X$ and~$y^{q} \in \langle x \rangle$ for~some \hbox{$\pi^\prime$-}num\-ber $q$, then the~subgroup $\langle x, y \rangle $ is cyclic.
\end{lemma}

\begin{proof}
First of~all let observe that $\mathcal{I}_{\pi^\prime}(X, \langle x \rangle)$ is an~abelian subgroup by~the~previous proposition and~$y \in \mathcal{I}_{\pi^\prime}(X, \langle x \rangle )$. Therefore, $[x, y]=1$.

Let $y^{q}=x^{k}$. Without lost of~generality we can consider $q$ to~be prime. So the~only two cases are possible: $q\vert k$ and~$(k, q)=1$.

If~$k=ql$, then the~equality $y^{q}=x^{ql}$ implies $(y^{-1}x^{l})^{q}=1$. But~$X$ is residually an~\hbox{$\mathcal{F}_{\pi}$-}group, so it is \hbox{$\pi^\prime$-}torsion-free by~Proposition~\ref{prop13}. Therefore, $y=x^{l}$ and~$\langle x, y \rangle = \langle x \rangle$.

If~$(k, q)=1$, then $ku + qv = 1$ for~certain whole numbers~$u$,~$v$ and
\begin{gather*}
x=x^{ku+qv}=y^{qu}x^{qv}=(y^{u}x^{v})^{q},\\ y=y^{ku+qv}=y^{ku}x^{kv}=(y^{u}x^{v})^{k}.
\end{gather*}

Thus, $\langle x, y \rangle = \langle y^{u}x^{v} \rangle$.
\end{proof}

\begin{proof}[\indent {\it Proof of~proposition}]
Let $u$, $v$ be arbitrary elements of~$\mathcal{I}_{\pi^\prime}(X, Y)$. By~Proposition~\ref{prop14}, $\mathcal{I}_{\pi^\prime}(X, Y)$ coincides with~the~set of~all \hbox{$\pi^\prime$-}roots being extracted from the~elements of~$Y$ in~$X$. So $u^{q}=y_{1} \in Y$, \hbox{and~$v^{r}=y_{2} \in Y$} for~some \hbox{$\pi^\prime$-}num\-bers~$q$ and~$r$.

Since $Y$ is locally cyclic, the~subgroup $\langle y_{1}, y_{2} \rangle$ is cyclic and~is generated by~an~element $y \in Y$. By~Lemma, the~subgroup $\langle y, u \rangle$ is also cyclic and~is generated by~an~element $w \in \mathcal{I}_{\pi^\prime}(X, Y)$. If~we apply now Lemma to~$w$ and~$v$, then we get that $\langle w, v \rangle = \langle t \rangle$ for~some element $t \in \mathcal{I}_{\pi^\prime}(X, Y)$. Thus, $u, v \in \langle t \rangle$, and~the~proposition is proved.
\end{proof}

\section{The~main theorem}

To formulate the~main result of~this paper we need some auxiliary designations.

If~$Y$ is a~subgroup of~a~group~$X$, then we denote by~$\Delta_{\pi}(X)$ the~family of~all \hbox{$\pi^\prime$-}iso\-lated cyclic subgroups of~this group which are not \hbox{$\mathcal{F}_{\pi}$-}separ\-able in~it. By~Proposition~\ref{prop13}, if~a~subgroup is not \hbox{$\pi^\prime$-}iso\-lated in~$X$, then it is a~fortiori not~\hbox{$\mathcal{F}_{\pi}$-}separ\-able in~this group. So the~family of~all \hbox{$\mathcal{F}_{\pi}$-}separ\-able cyclic subgroups of~$X$ is maximal if~$\Delta_{\pi}(X)=\varnothing$.

Let
$$
G = \langle A*B;\ [H,K] = 1 \rangle
$$
be the~free product of~groups~$A$ and~$B$ with~commuting subgroups \hbox{$H \leqslant A$} and~\hbox{$K \leqslant B$} (these notations will be fixed till the~end of~the~paper). We put
$$
\Theta_{\pi}(HK) = \{(X \cap H)(Y \cap K) \mid X \in \Omega_{\pi}(A), Y \in \Omega_{\pi}(B)\}
$$
and denote by $\Lambda_{\pi}(HK)$ the~family of~all cyclic subgroups lying\linebreak and~\hbox{$\pi^\prime$-}iso\-lated in~$HK$ but being not~separable by~$\Theta_{\pi}(HK)$.

\begin{theorem}\label{thrm21}
Let $\pi$ be a~set of~prime numbers that coincides with~the~set of~all prime numbers or~is one-element, and~let at~least one of~the~following equivalent statements holds:

1) $A$ and~$B$ are residually \hbox{$\mathcal{F}_{\pi}$-}groups, $H$ and~$K$ are \hbox{$\mathcal{F}_{\pi}$-}separ\-able in~the~factors;

2) $G$ is residually an~\hbox{$\mathcal{F}_{\pi}$-}group.

Then a~\hbox{$\pi^\prime$-}iso\-lated cyclic subgroup of~$G$ is \hbox{$\mathcal{F}_{\pi}$-}separ\-able in~this group if, and~only if, it is not conjugated with~any subgroup of~the~family
$$
\Delta_{\pi}(A) \cup \Delta_{\pi}(B) \cup \Lambda_{\pi}(HK).
$$
\end{theorem}

\begin{proof}
First of~all we note that the~equivalence of~the~conditions~1 and~2 was stated in~\cite{li05}. So the~only thing that it is necessary to~prove is the~criterion of~the~\hbox{$\mathcal{F}_{\pi}$-}separability of~a~cyclic subgroup of~$G$. We begin with~the~check of~the~necessity of~the~condition.

If~a~cyclic subgroup~$C$ belongs to~$\Delta_{\pi}(A)$, then it is not \hbox{$\mathcal{F}_{\pi}$-}separ\-able in~$A$ and~there exists an~element $a \in A\backslash C$ such that $a \in CX$ for~any subgroup $X \in \Omega_{\pi}(A)$. But~then $a \in CL$ for~every subgroup $L \in \Omega_{\pi}(G)$ since $L \cap A \in \Omega_{\pi}(A)$. Therefore, $C$ is not \hbox{$\mathcal{F}_{\pi}$-}separ\-able in~$G$. It is obvious that any subgroup conjugated with~$C$ is also not \hbox{$\mathcal{F}_{\pi}$-}separ\-able in~$G$.

It can be proved in~precisely the~same way that any subgroup\linebreak of~$\Delta_{\pi}(B)$ or~of~$\Lambda_{\pi}(HK)$ is not \hbox{$\mathcal{F}_{\pi}$-}separ\-able in~$G$. It needs only note that, if~$L \in \Omega_{\pi}(G)$, then
$$
(L \cap H)(L \cap K) \in \Theta_{\pi}(HK).
$$

Before passing to~the~sufficiency we give some facts about the~structure and~properties of~free products with~commuting subgroups.

Recall that \textit{the free product of~groups~$M$ and~$N$ with~subgroups\linebreak $U \leqslant M$ and~$V \leqslant N$ amalgamated according to~an~isomorphism\linebreak $\varphi\colon U \to V$,}
$$
T = \langle M*N;\ U = V,\ \varphi \rangle,
$$
is the~quotient group of~the~ordinary free product~$M*N$ of~$M$ and~$N$ by~the~normal closure of~the~set $\{u(u\varphi)^{-1} \mid u \in U\}$.

It is known that the~subgroups of~$T$ generated by~the~generators of~$M$ and~$N$ are isomorphic to~these groups and, therefore, may be identified with~them. Herewith, $U$ and~$V$ turn out to~be coincident and~this lets us to~write down $T$ in~the~form
$$
T=\langle M*N;\ U \rangle
$$
considering $\varphi$ as~given.

We need also one more simply checked property of~elements of~generalized free products of~two groups. Recall that \textit{the~length of~an~element} of~$T$ is the~length of~any reduced form of~this element.

\begin{proposition}[{\cite[Proposition~2.1.6]{li11}}]\label{prop22}
For every two elements $x, y \in T$, if~one of~them has an~even length and~$y=x^{q}$ for~some\linebreak positive number~$q$, then the~other element has an~even length too\linebreak \hbox{and~$l(y)=l(x)q$.} \qed
\end{proposition}

It is easy to~see~\cite{li05} that the~structure of~$G$ can be described in~terms of~the~construction of~free product with~amalgamated subgroup as~follows.

Let $U=HK$, $M=\langle A, U \rangle$, and~$N=\langle B, U \rangle$. Then $U$ is the~direct product $H\times K$ of~$H$ and~$K$, $M$ is the~free product $\langle A * U;\ H \rangle$ of~$A$ and~$U$ with~$H$ amalgamated, $N$ is the~free product $\langle B * U;\ K \rangle$ of~$B$ and~$U$ with~$K$ amalgamated, and~$G$ is the~free product $\langle M * N;\ U\rangle$ of~$M$ and~$N$ with~$U$ amalgamated.

If $Y$ is a~subgroup of~a~group~$X$, we put
$$
\Omega_{\pi}(X, Y)=\{L \cap Y \mid L \in \Omega_{\pi}(X)\}
$$
and denote by~$\Delta_{\pi}(X, Y)$ the~family of~all cyclic subgroups lying\linebreak and~\hbox{$\pi^\prime$-}iso\-lated in~$Y$, but being not separable by~$\Omega_{\pi}(X, Y)$. The following statements take place.

\begin{proposition}[{\cite[Theorems 1.2 and~1.6]{li12}}]\label{prop23}
Let $\pi$ be a~set of~prime numbers which coincides with~the~set of~all prime numbers or~is one-element, and~let $T = \langle M*N;\ U \rangle$ be the~free product of~groups~$M$ and~$N$ with~an~amalgamated subgroup~$U$. If~$\{1\}$ and~$U$ are separable by~both~$\Omega_{\pi}(T, M)$ and~$\Omega_{\pi}(T, N)$, then a~\hbox{$\pi^\prime$-}iso\-lated cyclic subgroup of~$T$ is \hbox{$\mathcal{F}_{\pi}$-}separ\-able in~this group if, and~only if, it is not conjugated with~any subgroup of~$\Delta _{\pi}(T, M) \cup \Delta _{\pi}(T, N)$. \qed
\end{proposition}

\begin{proposition}[{\cite[Lemmas 1 and~3]{li05}}]\label{prop24}
Let $\pi$ be a~set of~prime numbers which coincides with~the~set of~all prime numbers or~is one-element. If~$A$, $B$ are residually \hbox{$\mathcal{F}_{\pi}$-}groups and~$H$, $K$ are \hbox{$\mathcal{F}_{\pi}$-}separ\-able in~the~free factors, then $\{1\}$ and~$U$ are separable by~both~$\Omega_{\pi}(G, M)$ and~$\Omega_{\pi}(G, N)$. \qed
\end{proposition}

It follows also from the~proof of~Lemmas~1 and~3 of~\cite{li05} that the~statement below holds.

\begin{proposition}\label{prop25}
Let $\pi$ be a~set of~prime numbers which coincides with~the~set of~all prime numbers or~is one-element, \hbox{$X \in \Omega_{\pi}(A)$,} \hbox{$Y \in \Omega_{\pi}(B)$,} and~$Q=(X \cap H)(Y \cap K)$. Then any subgroup $R \in \Omega_{\pi}(M)$ such that $R \cap A=X$ and~$R \cap U=Q$ belongs to~$\Omega_{\pi}(G, M)$. The~same statement holds for~the~subgroups of~$\Omega_{\pi}(N)$. \qed
\end{proposition}

Now we can turn straight to~proof of~sufficiency. Due to~Propositions~\ref{prop23} and~\ref{prop24}, we need only show that every subgroup of~$\Delta_{\pi}(G, M)$ is conjugated with~some subgroup of~$\Delta_{\pi}(A) \cup \Lambda_{\pi}(U)$ and~every subgroup of~$\Delta_{\pi}(G, N)$ does with~some subgroup of~$\Delta_{\pi}(B) \cup \Lambda_{\pi}(U)$. We consider $M$, the~arguments for~$N$ are the~same.

Subgroups $X \leqslant A$ and~$Q \leqslant U$ will be called \textit{$(H, \pi)$-com\-patible} if~there exists a~subgroup $L \in \Omega_{\pi}(M)$ such that \hbox{$L \cap A=X$} and\linebreak $L \cap U=Q$. Since in~this case $X \cap H=Q \cap H$, the~function
$$
\varphi_{X, Q}\colon HX/X \to HQ/Q
$$
which maps an~element $hX$, $h \in H$, to~$hQ$ is a~correctly defined isomorphism of~subgroups. Therefore, we can construct the~group
$$
M_{X,Q} = \langle A/X*U/Q;\ HX/X = HQ/Q,\ \varphi_{X, Q} \rangle
$$
and extend the~natural homomorphisms of~$A$ onto~$A/X$ and~of~$U$ onto~$U/Q$ to~the~homomorphism $\rho_{X, Q}$ of~$M$ onto~$M_{X, Q}$.

Note that to~prove the~separability of~a~cyclic subgroup~$C$ of~$M$ by~$\Omega_{\pi}(G, M)$ it is sufficient to~be able to~find for~every element \hbox{$g \in M\backslash C$} a~pair of~subgroups $X \in \Omega_{\pi}(A)$ and~$Y \in \Omega_{\pi}(B)$ such that $g\rho_{X, Q}$ doesn't belong to~some \hbox{$\pi^\prime$-}iso\-lated cyclic subgroup $D_{X, Q}$ of~$M_{X, Q}$ containing $C\rho_{X, Q}$ (where, as~early, $Q=(X \cap H)(Y \cap K))$.

Indeed, it is proved in~\cite{li12} that the~$(H, \pi)$-compatibility of~$X$ and~$Q$ implies that $M_{X, Q}$ is residually an~\hbox{$\mathcal{F}_{\pi}$-}group and~all its \hbox{$\pi^\prime$-}iso\-lated cyclic subgroups are \hbox{$\mathcal{F}_{\pi}$-}separ\-able. So, if~$X$ and~$Y$ possess the~mentioned properties, then there exists a~subgroup $R_{X, Q} \in \Omega_{\pi}(M_{X, Q})$ such that $g\rho_{X, Q} \notin D_{X, Q}R_{X, Q}$ and, hence, $g\rho_{X, Q} \notin C\rho_{X, Q}R_{X, Q}$.

Further, since the~quotient groups~$A/X$ and~$U/Q$ are finite, we can consider that
$$
R_{X, Q} \cap A/X=R_{X, Q} \cap U/Q=1.
$$

Then the~pre-image~$R$ of~$R_{X, Q}$ under~$\rho_{X, Q}$ satisfies the~relations \hbox{$R \cap A=X$} \hbox{and~$R \cap U=Q$.} It follows now from Proposition~\ref{prop25} that $R \in \Omega_{\pi}(G, M)$ and~meanwhile $g \notin CR$ as~it is required.

So let $C=\langle c\rangle $ be a~\hbox{$\pi^\prime$-}iso\-lated cyclic subgroup of~$M$ being not conjugated with~any subgroup of~$\Delta_{\pi}(A) \cup \Lambda_{\pi}(U)$. We show that $C$ is separable by~$\Omega_{\pi}(G, M)$.

Let $g \in M$ be an~arbitrary element not belonging to~$C$, and~let 
$$
g=g_{1}g_{2}\ldots g_{m},\ c=c_{1}c_{2}\ldots c_{n}
$$
be reduced forms of~$g$ and~$c$ considering as~the~elements of~the~generalized free product~$M$. Applying, if~necessary, a~suitable inner automorphism of~$G$ defining by~an~element of~$M$ we can consider further that $c$ is cyclically reduced.

Let, at~first, $n=1$, and~let, for~definiteness, $c \in A$ (the case \hbox{when~$c \in U$} is considered in~the~same way).

Note that if~$X \in \Omega_{\pi}(A)$, $Y \in \Omega_{\pi}(B)$, and~$Q=(X \cap H)(Y \cap K)$, then all cyclic subgroups of~the~free factors of~$M_{X, Q}$ (the finite \hbox{$\pi$-}groups) are \hbox{$\pi^\prime$-}iso\-lated. So it is sufficient to~point out a~pair of~subgroups $X \in \Omega_{\pi}(A)$, $Y \in \Omega_{\pi}(B)$ such that $g\rho_{X, Q} \notin C\rho_{X, Q}$.

By the~condition, $C$ is separable by~$\Omega_{\pi}(A)$. Therefore, if~$g \in A$, then there exists a~subgroup $X \in \Omega_{\pi}(A)$ such that $g \notin CX$ and, hence, $g\rho_{X, Q} \notin C\rho_{X, Q}$ regardless of~the~choice of~$Y$.

Let $g \notin A$. Then, if~$m=1$, $g=g_{1} \in U\backslash H$. When $m>1$, every syllable $g_{i}$ of~the~reduced form of~$g$ belongs to~one of~the~free factors and~doesn't belong to~the~amalgamated subgroup.

By the~condition of~the~theorem, $H$ is separable by~$\Omega_{\pi}(A)$.

If $u \in U\backslash H$ and~$u=hk$ for~suitable elements $h \in H$, $k \in K$, then $k \ne 1$ and, since $B$ is residually an~\hbox{$\mathcal{F}_{\pi}$-}group, there exists a~subgroup $Y \in \Omega_{\pi}(B)$ such that $k \notin Y$. Then 
$$
u \notin H(X \cap H)(Y \cap K)=H(Y \cap K)
$$
regardless of~the~choice of~a~subgroup $X \in \Omega_{\pi}(A)$. Thus, $H$ is separable by~$\Theta_{\pi}(U)$.

It follows from the~told that, for~every~$i$ ($1 \leqslant i \leqslant m$), one can point out a~pair of~subgroups $X_{i} \in \Omega_{\pi}(A)$, $Y_{i} \in \Omega_{\pi}(B)$ such that $g_{i} \notin HX_{i}$ if~$g_{i} \in A$ and~$g_{i} \notin H(Y_{i} \cap K)$ if~$g_{i} \in U$. Let us put
$$
X = \bigcap_{i = 1}^m {X_i},\ Y = \bigcap_{i = 1}^m {Y_i}.
$$

It is obvious that then $X \in \Omega_{\pi}(A)$, $Y \in \Omega_{\pi}(B)$ and~the~presentation of~$g\rho_{X, Q}$ as~the~product
$$
g\rho_{X, Q}=g_{1}\rho_{X, Q}g_{2}\rho_{X, Q}\ldots g_{m}\rho_{X, Q}
$$
is still the~reduced form in~$M_{X, Q}$. Therefore, its length $l(g\rho_{X, Q})$ equals the~length of~$g$, and, if~$m=1$, then $g\rho_{X, Q} \in U\rho_{X, Q}\backslash H\rho_{X, Q}$. Thus, in~this case the~element $g\rho_{X, Q}$ doesn't belong to~$C\rho_{X, Q}$ too.

Let now $n \geqslant 2$. As above, we find a~pair of~subgroups $X \in \Omega_{\pi}(A)$, $Y \in \Omega_{\pi}(B)$ such that $l(g\rho_{X, Q})=l(g)$ and~$l(c\rho_{X, Q})=l(c)$. Note that $c\rho_{X, Q}$ is still a~cyclically reduced element.

For any pair of~subgroups $V \in \Omega_{\pi}(A)$, $W \in \Omega_{\pi}(B)$ such that $V \leqslant X$ and~$W \leqslant Y$, the~equality $l(c\rho_{V, P})=l(c)$ holds, where
$$
P=(V \cap H)(W \cap K).
$$

So the~powers of~the~roots which can be extracted from~$c\rho_{V, P}$ are bounded as~a~whole by~Proposition~\ref{prop22}. It follows now from Proposition~\ref{prop15} that $\mathcal{I}_{\pi^\prime}(M_{V, P}, C\rho_{V, P})$ is a~cyclic subgroup. Show that $V$\linebreak and~$W$ can be chosen in~such a~way that $g\rho_{V, P}$ doesn't belong\linebreak to~$\mathcal{I}_{\pi^\prime}(M_{V, P}, C\rho_{V, P})$.

Write $n$ in~the~form $n=qt$, where~$q$ is a~\hbox{$\pi$-}num\-ber, $t$ is a~\hbox{$\pi^\prime$-}num\-ber if~$\pi$ doesn't coincide with~the~set of~all prime numbers and~$t=1$ otherwise.

\textit{Case~1.} $n$ doesn't divide~$mt$.

Since $n$ doesn't divide $mt$,\,then,\,by~Proposition~\ref{prop22},\,\hbox{$(g\rho_{X, Q})^{t} \notin C\rho_{X, Q}$.} Show that $g\rho_{X, Q} \notin \mathcal{I}_{\pi^\prime}(M_{X, Q}, C\rho_{X, Q})$.

Let $d_{X, Q}$ denote a~generator of~$\mathcal{I}_{\pi^\prime}(M_{X, Q}, C\rho_{X, Q})$, and~let ($d_{X, Q})^{z}=c\rho_{X, Q}$. It follows from Proposition~\ref{prop22} that then $z\vert n$. But~$z$ is a~\hbox{$\pi^\prime$-}num\-ber, so it divides $t$ and, therefore, 
$$
(\mathcal{I}_{\pi^\prime}(M_{X, Q}, C\rho_{X, Q}))^{t} \leqslant C\rho_{X, Q}.
$$
Thus, supposing that $g\rho_{X, Q} \in \mathcal{I}_{\pi^\prime}(M_{X, Q}, C\rho_{X, Q})$ we come to~the~inclusion $(g\rho_{X, Q})^{t} \in C\rho_{X, Y}$, which contradicts the~relation stated above.

\textit{Case~2.} $mt=nk$ for~certain positive~$k$.

Since $C$ is \hbox{$\pi^\prime$-}iso\-lated in~$M$ and~$g \notin C$, then $g^{t} \ne c^{\pm k}$. Arguing as~above we find subgroups $R \in \Omega_{\pi}(A)$, $S \in \Omega_{\pi}(B)$ such that
$$
(g^{-t}c^{k})\rho_{R, L} \ne 1 \ne (g^{-t}c^{-k})\rho_{R, L},
$$
where~$L=(R \cap H)(S \cap K)$. Let us put $V=X \cap R$, $W=Y \cap S$, and~$P=(V \cap H)(W \cap K)$.

Then $(g\rho_{V, P})^{t} \ne (c\rho_{V, P})^{\pm k}$, and, since 
$$
l(g\rho_{V, P})=l(g)=m,\ l(c\rho_{V, P})=l(c)=n,
$$
we have $(g\rho_{V, P})^{t} \notin C\rho_{V, P}$. As~well as~in~the~case discussed above it follows that $g\rho_{V, P} \notin \mathcal{I}_{\pi^\prime}(M_{V, P}, C\rho_{V, P})$, and~the~proof is finished.
\end{proof}

\section{Some corollaries}

Unfortunately, to~describe the~family $\Lambda_{\pi}(HK)$ is difficult in~general case, and~our further efforts will be directed on~making it at~least in~some special cases.

\begin{proposition}\label{prop31}
If $A$ is \hbox{$\mathcal{F}_{\pi}$-}quasi-regu\-lar by~$H$ and~$B$ is \hbox{$\mathcal{F}_{\pi}$-}quasi-regu\-lar by~$K$, then $\Lambda_{\pi}(HK)=\Delta_{\pi}(HK)$.
\end{proposition}

\begin{proof}
Indeed, if~$L \in \Omega_{\pi}(HK)$, then we can use the~\hbox{$\mathcal{F}_{\pi}$-}quasi-regu\-lari\-ty and~find such subgroups $X \in \Omega_{\pi}(A)$, $Y \in \Omega_{\pi}(B)$ that $X \cap H \leqslant L \cap H$, $Y \cap K \leqslant L \cap K$ and, hence, $(X \cap H)(Y \cap K) \leqslant L$. Therefore, any subgroup lying in~$HK$ and~being separable by~$\Omega_{\pi}(HK)$ turns out to~be separable by~$\Theta_{\pi}(HK)$ too. Since the~inverse statement is obvious, $\Lambda_{\pi}(HK)=\Delta_{\pi}(HK)$ as~it is required.
\end{proof}

The~next corollary follows directly from Propositions~\ref{prop31} and~\ref{prop12}.

\begin{corollary}\label{corl32}
Let $H$, $K$ be subnormal in~$A$, $B$ and~have finite \hbox{$\pi$-}indexes in~them. Then $\Lambda_{\pi}(HK)=\Delta_{\pi}(HK)$. \qed
\end{corollary}

Alas! the~structure of~the~set $\Delta_{\pi}(HK)$ is also not~simple. And~so we continue the~research of~the~separability of~the~cyclic subgroups lying in~$HK$ and~prove

\begin{proposition}\label{prop33}
Let $A$ and~$B$ be residually \hbox{$\mathcal{F}_{\pi}$-}groups, and~let $C$ be a~\hbox{$\pi^\prime$-}iso\-lated cyclic subgroup of~$HK$ generated by~an~element $c=hk$, where~$h \in H$, $k \in K$. If~$h$ has infinite order and~all subgroups of~$\Omega_{\pi}(\mathcal{I}_{\pi^\prime}(A, \langle h \rangle ))$ are \hbox{$\mathcal{F}_{\pi}$-}separ\-able in~$A$, then $C$ is separable by~$\Theta_{\pi}(HK)$.
\end{proposition}

\begin{proof}
Let $g \in HK\backslash C$ be an~arbitrary element. We need to~point~out such subgroups $X \in \Omega_{\pi}(A)$, $Y \in \Omega_{\pi}(B)$ that $g \notin C(X \cap H)(Y \cap K)$.

Write $g$ in~the~form $g=ab$ for~suitable elements $a \in H$, $b \in K$.

If $a \notin \Omega_{\pi}(A)$\hbox{-\textit{Cl}}$(\langle h \rangle)$, then there exists a~subgroup $X \in \Omega_{\pi}(A)$ such that $a \notin \langle h \rangle X$. Then $ab \notin C(X \cap H)K$ since $HK$ is the~direct product of~$H$ and~$K$, and~we can take $B$ wholly as~$Y$. The~case when~$b \notin \Omega_{\pi}(B)$\hbox{-\textit{Cl}}$(\langle k \rangle)$ is considered in~the~same~way.

Let now $a \in \Omega_{\pi}(A)$\hbox{-\textit{Cl}}$(\langle h \rangle)$ and~$b \in \Omega_{\pi}(B)$\hbox{-\textit{Cl}}$(\langle k \rangle)$.

Since $A$ is residually an~\hbox{$\mathcal{F}_{\pi}$-}group, then $\mathcal{I}_{\pi^\prime}(A, \langle h \rangle)$ is locally cyclic according to~Proposition~\ref{prop15}. By~the~condition, this subgroup is \hbox{$\mathcal{F}_{\pi}$-}separ\-able in~$A$, hence, $\Omega_{\pi}(A)$\hbox{-\textit{Cl}}$(\langle h \rangle)=\mathcal{I}_{\pi^\prime}(A, \langle h \rangle)$ and~$a \in \mathcal{I}_{\pi^\prime}(A, \langle h \rangle)$. Denote by~$u$ the~generator of~the~subgroup $\langle a, h \rangle$ and~consider two cases.

\textit{Case~1.} $\langle k \rangle$ is finite.

Since $B$ is residually an~\hbox{$\mathcal{F}_{\pi}$-}group, every its finite subgroup is\linebreak \hbox{$\mathcal{F}_{\pi}$-}separ\-able. Therefore, $b \in \Omega_{\pi}(B)$\hbox{-\textit{Cl}}($\langle k \rangle)=\langle k \rangle$ and~$C$ is a~\hbox{$\pi^\prime$-}iso\-lated cyclic subgroup of~the~finitely generated abelian group $\langle u, k \rangle$.

It is known \cite{li05} that all \hbox{$\pi^\prime$-}iso\-lated subgroups of~a~finitely generated nilpotent group are \hbox{$\mathcal{F}_{\pi}$-}separ\-able. So there exists a~subgroup \hbox{$L \in \Omega_{\pi}(\langle u, k \rangle)$} such that $g \notin CL$. 

Note that for~any \hbox{$\pi$-}num\-ber~$l$ $\mathcal{I}_{\pi^\prime}(A, \langle h \rangle)^{l} \cap \langle u\rangle \leqslant \langle u^{l} \rangle$.

Indeed, let $v \in \mathcal{I}_{\pi^\prime}(A, \langle h\rangle)$ and~$v^{l} \in \langle u \rangle$. By~Proposition~\ref{prop14}, $\mathcal{I}_{\pi^\prime}(A, \langle h \rangle)$ coincides with~the~set of~all \hbox{$\pi^\prime$-}roots being extracted from~$\langle h \rangle$ in~$A$. So there exists a~\hbox{$\pi^\prime$-}num\-ber $m$ such that $v^{m} \in \langle h \rangle \leqslant \langle u \rangle$. Since $l$ and~$m$ are relatively prime, the~equality $l\alpha + m\beta = 1$ holds for~suitable integers $\alpha$ and~$\beta$. From~this it follows that $v=v^{l\alpha + m\beta} \in \langle u \rangle$ and~$v^{l} \in \langle u^{l} \rangle$.

Let now $l=[\langle u \rangle : L \cap \langle u \rangle]$. By~Proposition~\ref{prop11}, $A$ is \hbox{$\mathcal{F}_{\pi}$-}quasi-regu\-lar by~$\mathcal{I}_{\pi^\prime}(A, \langle h \rangle)$. We use this fact and~choose a~subgroup \hbox{$X \in \Omega_{\pi}(A)$} so that $X \cap \mathcal{I}_{\pi^\prime}(A, \langle h \rangle) \leqslant \mathcal{I}_{\pi^\prime}(A, \langle h \rangle)^{l}$. Since $B$ is residually an~\hbox{$\mathcal{F}_{\pi}$-}group, there exists also a~subgroup $Y \in \Omega_{\pi}(B)$ satisfying the~condition\linebreak \hbox{$Y \cap \langle k \rangle=1$.}

Suppose that $g \in C(X \cap H)(Y \cap K)$ and~$g=c^{n}xy$ for~suitable elements $x \in X \cap H$, $y \in Y \cap K$ and~a~number~$n$. Since $g, c \in \langle u, k \rangle$ and~$a, h \in \mathcal{I}_{\pi^\prime}(A, \langle h \rangle)$, then
\begin{gather*}
x \in X \cap H \cap \langle u, k\rangle \cap \mathcal{I}_{\pi^\prime}(A, \langle h \rangle) = X \cap \mathcal{I}_{\pi^\prime}(A, \langle h \rangle) \cap \langle u \rangle \leqslant \\
\leqslant \mathcal{I}_{\pi^\prime}(A, \langle h \rangle)^{l} \cap \langle u \rangle \leqslant \langle u^{l} \rangle = L \cap \langle u \rangle
\end{gather*}
and 
$$
y \in Y \cap K \cap \langle u, k \rangle = Y \cap \langle k \rangle = 1,
$$
whence $g \in CL$ what contradicts the~choice of~$L$. Thus, 
$$
g \notin C(X \cap H)(Y \cap K)
$$
as~it is required.

\textit{Case~2.} $\langle k \rangle$ is infinite.

Let $a=u^{l}$ and~$h=u^{q}$. Since $\mathcal{I}_{\pi^\prime}(A, \langle h \rangle)$ coincides with~the~set of~all \hbox{$\pi^\prime$-}roots being extracted from $\langle h \rangle$ in~$A$, we can consider $q$ as~a~\hbox{$\pi^\prime$-}num\-ber.

Suppose that $b^{q}=k^{l}$. Then ($ab)^{q}=u^{lq}b^{q}=h^{l}k^{l}=(hk)^{l}$. But~$C$ is \hbox{$\pi^\prime$-}iso\-lated, so $q=1$ and~$g \in C$ what is impossible. Therefore, $b^{q}k^{-l} \ne 1$ and, since $B$ is residually an~\hbox{$\mathcal{F}_{\pi}$-}group, there exists a~subgroup \hbox{$Z \in \Omega_{\pi}(B)$} which doesn't contain this element.

Consider the~group 
$$
G_{Z} = \langle A*B/Z;\ [H,KZ/Z] = 1 \rangle.
$$

It is obvious that the~natural homomorphism of~$B$ onto~$B/Z$ can be extended to~the~surjective homomorphism $\rho_{Z}\colon G \to G_{Z}$. Let 
$$
D = \mathcal{I}_{\pi^\prime}(\langle u, kZ \rangle, C\rho_{Z}).
$$

Since $D$ lies in~the~finitely generated abelian group, it is cyclic and~is generated by~an~element~$d$. Write it in~the~form $d=x \cdot yZ$ for~suitable elements $x \in H$, $yZ \in KZ/Z$.

Let $s$ be a~\hbox{$\pi^\prime$-}num\-ber such that $d^{s} \in C\rho_{Z}$. Then $x^{s} \in \langle h\rangle $ and~so $x \in \mathcal{I}_{\pi^\prime}(A, \langle h \rangle)$. From~this it follows that $\mathcal{I}_{\pi^\prime}(A, \langle x \rangle) = \mathcal{I}_{\pi^\prime}(A, \langle h \rangle)$.

Suppose that $g\rho_{Z} \in D$. Then $(g\rho_{Z})^{r}=(c\rho_{Z})^{m}$ for~some numbers~$r$ and~$m$, while $r$ can be considered without lost of~generality as~a~\hbox{$\pi^\prime$-}num\-ber. Since $b \in \Omega_{\pi}(B)$\hbox{-\textit{Cl}}$(\langle k \rangle)$, then $b \in \langle k \rangle Z$ and~$b \equiv k^{t} \pmod{Z}$ for~a~suitable number~$t$. We have
$$
(u^{q} \cdot kZ)^{m} = (h \cdot kZ)^{m} = (c\rho_{Z})^{m} =  (g\rho_{Z})^{r} = (a \cdot bZ)^{r} = (u^{l} \cdot (kZ)^{t})^{r}
$$
whence $m \equiv tr \pmod{n}$, where~$n$ is the~order of~$kZ$, and~$qm=lr$ due to~infinity of~the~orders of~$h$ and~$u$.

Since $Z$ has a~finite \hbox{$\pi$-}index in~$B$, $n$ is a~\hbox{$\pi$-}num\-ber. Therefore, the~congruencies $qtr \equiv qm \equiv lr \pmod{n}$ imply that $qt \equiv l \pmod{n}$. But in~this case
$$
1 \equiv k^{tq-l} \equiv b^{q}k^{-l} \pmod{Z}
$$
what contradicts the~choice of~$Z$.

Thus, $g\rho_{Z} \notin D$, and, while we don't assert that $D$ is \hbox{$\pi^\prime$-}iso\-lated in~$H \cdot KZ/Z$, we can however repeat for~$G_{Z}$ the~same arguments as~in~the~case~1. As~a~result one can find in~this group such subgroups $X \in \Omega_{\pi}(A)$, $Y/Z \in \Omega_{\pi}(B/Z)$ that 
$$
g\rho_{Z} \notin D(X \cap H)(Y/Z \cap KZ/Z).
$$
It is obvious that $X$ and~$Y$ turn~out to~be required, and~the~proposition is proved.
\end{proof}

Note that with\,$H=A$\;and\,$K=B$\,the~equalities\,\hbox{$\Theta_{\pi}(HK)=\Omega_{\pi}(HK)$} and~$\Lambda_{\pi}(HK)=\Delta_{\pi}(HK)$ hold. So Proposition~\ref{prop33} and~the~corollary given below can be used also for~describing of~\hbox{$\mathcal{F}_{\pi}$-}separ\-able cyclic subgroups of~an~arbitrary direct product of~groups.

\begin{corollary}\label{corl34}
Let the~following statements hold:

1) $A$ and~$B$ are residually \hbox{$\mathcal{F}_{\pi}$-}groups,

2) $H$ and~$K$ are \hbox{$\pi^\prime$-}iso\-lated in~the~free factors,

3) all cyclic subgroups which lie in~$H$ and~are \hbox{$\pi^\prime$-}iso\-lated in~$A$ are \hbox{$\mathcal{F}_{\pi}$-}separ\-able in~$A$,

4) all cyclic subgroups which lie in~$K$ and~are \hbox{$\pi^\prime$-}iso\-lated in~$B$ are \hbox{$\mathcal{F}_{\pi}$-}separ\-able in~$B$,

5) $H$ and~$K$ don't contain locally cyclic subgroups unless $\pi$ coincides with~the~set of~all prime numbers.

Then $\Lambda_{\pi}(HK)=\varnothing$.
\end{corollary}

\begin{proof}
Let $C$ be a~\hbox{$\pi^\prime$-}iso\-lated cyclic subgroup of~$HK$ generated\linebreak by~an~element $c=hk$, where~$h \in H$, $k \in K$.

Let $C$ be finite, and~let $g=ab$, $a \in H$, $b \in K$, be an~arbitrary element which doesn't belong to~$C$. Since~$A$ and~$B$ are residually \hbox{$\mathcal{F}_{\pi}$-}finite, there exist subgroups $X \in \Omega_{\pi}(A)$, $Y \in \Omega_{\pi}(B)$ such that 
$$
X \cap \langle h \rangle = Y \cap \langle k \rangle = 1,
$$
$a^{-1}h^{m} \notin X$ whenever $a^{-1}h^{m} \ne 1$, and~$b^{-1}k^{n} \notin Y$ whenever $b^{-1}k^{n} \ne 1$. It is obvious that then $g \notin C(X \cap H)(Y \cap K)$ and~so $C$ is separable by~$\Theta_{\pi}(HK)$.

If $C$ is infinite, then at~least one of~the~elements~$h$ and~$k$ has an~infinite order. Let, for~definiteness, $h$ be such an~element.

Since $H$ is \hbox{$\pi^\prime$-}iso\-lated in~$A$, then $\mathcal{I}_{\pi^\prime}(A, \langle h \rangle) \leqslant H$. If~$\pi$ coincides with~the~set of~all prime numbers, then any subgroup is \hbox{$\pi^\prime$-}iso\-lated and~so $\mathcal{I}_{\pi^\prime}(A, \langle h \rangle)=\langle h \rangle$. Otherwise, by~condition, $H$ doesn't contain locally cyclic subgroups, and, therefore, $\mathcal{I}_{\pi^\prime}(A, \langle h \rangle)$ is cyclic too. But, again by~condition, any cyclic subgroup which lies in~$H$ and~is \hbox{$\pi^\prime$-}iso\-lated in~$A$ is \hbox{$\mathcal{F}_{\pi}$-}separ\-able in~$A$. Hence, $\mathcal{I}_{\pi^\prime}(A, \langle h \rangle)$ and~all its subgroups of~finite \hbox{$\pi$-}index are \hbox{$\mathcal{F}_{\pi}$-}separ\-able in~$A$.

Thus, we can use Proposition~\ref{prop33} in~this case, which states that $C$ is separable by~$\Theta_{\pi}(HK)$.
\end{proof}

The corollary proved in~a~combination with~the~main theorem lets generalize the~result of~E.~D.~Loginova~\cite{li06} that, if~$G$ is residually finite and~all cyclic subgroups of~$A$ and~$B$ are finitely separable in~these groups, then all cyclic subgroups of~$G$ are finitely separable too.

\begin{theorem}\label{thrm35}
Let $\pi$ be a~set of~prime numbers which coincides with~the~set of~all prime numbers or~is one-element. Let all cyclic subgroups which lie in~$H$ and~are \hbox{$\pi^\prime$-}iso\-lated in~$A$ be \hbox{$\mathcal{F}_{\pi}$-}separ\-able in~$A$, and~let all cyclic subgroups which lie in~$K$ and~are \hbox{$\pi^\prime$-}iso\-lated in~$B$ be \hbox{$\mathcal{F}_{\pi}$-}separ\-able in~$B$. Let also $H$ and~$K$ don't contain locally cyclic subgroups unless $\pi$ coincides with~the~set of~all prime numbers. If~$G$ is residually an~\hbox{$\mathcal{F}_{\pi}$-}group, then a~\hbox{$\pi^\prime$-}iso\-lated cyclic subgroup of~this group is \hbox{$\mathcal{F}_{\pi}$-}separ\-able in~it if, and~only if, it is not conjugated with~any subgroup of~$\Delta_{\pi}(A) \cup \Delta_{\pi}(B)$. \qed
\end{theorem}

It is not difficult to~show (see., e.~g., \cite{li13}) that every \hbox{$\pi^\prime$-}iso\-lated cyclic subgroup of~an~arbitrary free group is \hbox{$\mathcal{F}_{\pi}$-}separ\-able regardless of~the~choice of~a~set~$\pi$. As~it was already noted during the~proof of~Proposition~\ref{prop33} the~stronger assertion holds for~finitely generated nilpotent groups: all \hbox{$\pi^\prime$-}iso\-lated subgroups of~such groups are \hbox{$\mathcal{F}_{\pi}$-}separ\-able. So the~next two statements result from Theorem~\ref{thrm21} and~Corollary~\ref{corl34}.

\begin{theorem}\label{thrm36}
Let $\pi$ be a~set of~prime numbers which coincides with~the~set of~all prime numbers or~is one-element. Let $A$ and~$B$ be free groups, and~let $H$ and~$K$ be cyclic subgroups \hbox{$\pi^\prime$-}iso\-lated in~the~free factors. Then all \hbox{$\pi^\prime$-}iso\-lated cyclic subgroups of~$G$ are \hbox{$\mathcal{F}_{\pi}$-}separ\-able. \qed
\end{theorem}

\begin{theorem}\label{thrm37}
Let $\pi$ be a~set of~prime numbers which coincides with~the~set of~all prime numbers or~is one-element, and~let $A$ and~$B$ be finitely generated nilpotent groups. If~$G$ is residually an~\hbox{$\mathcal{F}_{\pi}$-}group, then all its \hbox{$\pi^\prime$-}iso\-lated cyclic subgroups are \hbox{$\mathcal{F}_{\pi}$-}separ\-able. \qed
\end{theorem}

Note that the~conditions of~Theorem~\ref{thrm36} are true for~the~group 
$$
G_{mn} = \langle a, b;\ [a^{m}, b^{n}] = 1 \rangle
$$
with~a~suitable choice of~numbers~$m$ and~$n$. This group is investigated in~\cite{li08, li07}, where~the~finite separability of~all its cyclic subgroups is proved.

\end{document}